\documentclass[12pt]{article}

\usepackage{amssymb}
\usepackage{amsthm}
\usepackage{xpatch}
\usepackage{amsmath}
\usepackage[utf8]{inputenc}
\usepackage{tikz,scalerel,tikz-cd}

\usepackage[english]{babel}
\usepackage{stmaryrd}
\xpatchcmd{\proof}{\itshape}{\prooflabelfont}{}{}
\newcommand{\prooflabelfont}{\bfseries}
\usepackage[colorlinks,citecolor=blue,urlcolor=blue,linkcolor=blue]{hyperref}
\newtheorem{thm}{Theorem}[section]

\newtheorem{lemma}[thm]{Lemma}
\newtheorem{cor}[thm]{Corollary}
\newtheorem{prop}[thm]{Proposition}
\newtheorem{claim}{Claim}[thm]
\newtheorem{problem}[thm]{Problem}

\newtheorem{question}{Question}[section]
\usepackage{quiver} 
\newtheorem{defn}[thm]{Definition}

\usepackage[alf]{abntex2cite}
\newtheorem{remark}[thm]{Remark}
\usepackage{geometry}
\usepackage{authblk}
 \usepackage{memhfixc}
 \usepackage{setspace}
 \usepackage{url}
\usepackage{enumitem}

\newcommand{\po}[2]{\langle #1 , #2 \rangle}

\usepackage{lineno}

\usepackage{tcolorbox}

\usepackage{authblk}

		\title{On totally Lindel\"of spaces}
		
		\newcommand{\Addresses}{{% additional braces for segregating \footnotesize
  \bigskip
  \footnotesize

  G.~Fernandes, \textsc{Departamento de Matemática - Instituto de Matemática e Estatística - USP, Rua do Matão, 1010, São Paulo, 05508-090, SP, Brazil}\par\nopagebreak
  \textit{E-mail address}: \texttt{gzfernandes@pm.me} \par\nopagebreak 
  \textit{URL}: \href{http://www.gabrielzf.com}{http://www.gabrielzf.com}

  \medskip

  G.~Pinto, \textsc{Departamento de Matemática - Instituto de Ciências Matemáticas e Computação - USP, Avenida Trabalhador São-carlense, 400, São Carlos, 13566-590, SP, Brazil }\par\nopagebreak
  \textit{E-mail address}: \texttt{guipullus.gp@usp.br}

   \medskip

  V.~Rocha, \textsc{Departamento de Matemática - Instituto de Matemática e Estatística - USP, Rua do Matão, 1010, São Paulo, 05508-090, SP, Brazil}\par\nopagebreak
  \textit{E-mail address}: \texttt{vorocha@ime.usp.br}

}}
		\author{Gabriel Fernandes, Guilherme Pinto\footnote{The author is funded by CAPES-
88887.829890/2023-00} ~and Vinicius Rocha \footnote{The author is funded by CAPES-
88887.829890/2023-00}}
\begin{document}	
		\maketitle

		\begin{abstract}
			%% Text of abstract
			The results in this paper answer three questions asked by \cite{noble} and give a partial answer to a question asked by  \cite{Alster}. We prove that every Alster space is totally Lindelöf and this gives a new characterization of regular Alster spaces. We construct a non-regular totally  Lindelöf space that is not Alster and we prove that there exists a Lindelöf $P$-space that is not Frolík.
		\end{abstract}

		%\begin{keyword}
			
			%filter bases \sep totally Lindelöf \sep Alster spaces \sep Frolík spaces
			
			%\MSC 54D20 \sep 54D10 \sep 	54D45
		%\end{keyword}

	%% \linenumbers
	
	%% main text

	\section{Introduction}

\textit{Totally Lindelöf} is a property of topological spaces that was defined in \cite{vaughan} by J. Vaughan  in order to unify the proof that countable products of Lindelöf P-spaces (see Definition \ref{Pspace}) and countable products of $\sigma$-compact spaces are Lindelöf (see \cite{vaughan} and Definition \ref{TotLindDef}). Alster defined in \cite{Alster} a property, nowadays known as  the Alster property, in an attempt to characterize the productively Lindelöf spaces (see  Definition \ref{AlsterDef}).

Since 2009, a series of results have been obtained on Alster spaces and related spaces and their relationship to the class of productively Lindelöf spaces. See \cite{noble}, \cite{tall1}, \cite{tall2}, \cite{tall3}, \cite{tall4}, \cite{tall5}, \cite{tall6}, \cite{zdomsky}, and \cite{aurichi}. Our main result in this paper is the following new characterization of Alster spaces:

\begin{thm} \label{AlstervsTotthm}
	If $X$ is a regular topological space, then $X$ is an Alster space if and only if $X$ is totally Lindelöf.
\end{thm}

The fact that regular totally Lindelöf spaces are Alster spaces is due to Noble in \cite[Proposition 5.1]{noble}. We will prove the other direction in Theorem \ref{AlsterImpTotthm}. Regarding totally Lindelöf spaces, Noble asked the following questions in \cite[Section 5]{noble}:
\begin{problem} \label{problem1}
\begin{itemize}
	\item[(a)] Is there an example of an Alster space that is not totally Lindelöf?
	\item[(b)] Is there an example of a totally Lindelöf (not regular) space that is not Alster?
	\item[(c)] What is the relationship between Alster (totally Lindelöf) spaces and Frolík spaces?
\end{itemize}
\end{problem}

%We will study totally Lindelöf spaces and filter bases and answer the three questions above. The paper is organized as follows:

In Section \ref{notation} below, we establish our notation and main definitions. In Section \ref{questiona}, we prove properties of filter bases, in particular, of filter bases generated by $G_{\delta}$ sets of a given compact $K$, and Theorem \ref{AlsterImpTotthm}, which answers Problem \ref{problem1}(a) negatively.

In Section \ref{questionb}, we show that $(\mathbb{R},\rho)$, a refinement of the usual topology of the reals, provides a positive answer to Problem \ref{problem1}(b), moreover, we show that such space $(\mathbb{R},\rho)$ is Hausdorff and Menger. 

In Section \ref{questionc}, we introduce the $\omega_{1}$-discrete set property (Definition \ref{dsp}), and show that Frolík spaces do not satisfy this property. Together with the results of the previous sections we will prove that the Lindelöfication of the discrete topology on $\omega_{1}$ gives a space that is Alster and is not Frolík, which together with the fact that $\omega^{\omega}$ is Frólik but it is not totally Lindelöf, answers Problem \ref{problem1}(c).

%\begin{defn}	If $X$ is a topological space and there is $\{x_{\alpha} \mid \alpha < \omega_{1}\} \subseteq X$ such that for all $\beta < \omega_1$, we have that $\{x_{\alpha} \mid \alpha < \beta \}$ is a discrete set. Then we say that $X$ has the $\omega_1$ discrete set property.\end{defn}

%Our main results in Section \ref{questionc} are Theorem \ref{ProdSigmathm} and Corollary \ref{Lw1NotFrolík}. Together with Theorem \ref{AlstervsTotthm}, they answer question (c):

%\begin{thm}\label{ProdSigmathm}	Let $\{X_n\mid n<\omega\}$ be a sequence of $\sigma$-compact spaces.	Then the product space $\prod_{n<\omega}X_n$ does not have the $\omega_{1}$ discrete set property.\end{thm}

%\begin{cor} \label{Lw1NotFrolík}
%	There is a Lindelöf P-space that is not Frolík, namely $L(\omega_1)$ is not Frolík.
%\end{cor}

	\section{Notation and Preliminaries \label{notation}}

We denote by $L(\omega_1)$ the topological space $(\omega_1+1,\tau)$ generated by the singletons $\{\alpha\}$, for  $\alpha < \omega_{1}$, and the intervals $]\alpha,\omega_{1}+1] = \{\beta \in \omega_{1} \mid \alpha < \beta \leq \omega_{1}+1\}$, for $\alpha < \omega_{1}$.

\begin{defn}
	If $X$ is a set and $\mathcal{F} \subseteq \mathcal{P}(X)$:
	\begin{enumerate}
		\item We say that $\mathcal{F}$ is a filter base on $X$ if for all $F_0,F_1 \in \mathcal{F}$, we have $F_{0} \cap F_{1} \in \mathcal{F}$ and $\emptyset \notin \mathcal{F}$.
		\item If $\mathcal{H}$ and $\mathcal{F}$ are filter bases on $X$, we say that $\mathcal{H}$ extends $\mathcal{F}$ if $\mathcal{F} \subseteq \mathcal{H}$.
	\end{enumerate}
\end{defn}

The definitions of totally Lindelöf presented in \cite{noble} and \cite{vaughan} seem to be different at first sight. In \cite{noble}, the definition of totally Lindelöf uses the concept of $\delta$-stable, while in \cite{vaughan}, it uses the concept of stable under countable intersections. It is not hard to see that both definitions are equivalent, but for  sake of self-containment we shall introduce both notions and show in Proposition \ref{stable} that they are equivalent definitions of totally Lindelöf.

\begin{defn}
	If $X$ is a set and $\mathcal{F}$ is a filter base on $X$:
	\begin{itemize}
		\item[(a)] We say that $\mathcal{F}$ is stable under countable intersections if for all $S \in [\mathcal{F}]^{\leq \omega}$, there is $H \in \mathcal{F}$ such that $H \subseteq \bigcap S$.
		\item[(b)] We say that $\mathcal{F}$ is $\omega$-closed if for all $S \in [\mathcal{F}]^{\leq \omega}$, we have $\bigcap S \in \mathcal{F}$.
		\item[(c)] We say that $\mathcal{F}$ is $\delta$-stable if for all $S \in [\mathcal{F}]^{\leq \omega}$, we have $\bigcap S \neq \emptyset$.
	\end{itemize}
\end{defn}

\begin{defn}
	If $\mathcal{F} \subseteq \mathcal{P}(X)$ and $(X,\tau)$ is a topological space, we define $ad^{(X,\tau)}(\mathcal{F})= \bigcap \{\bar{F} \mid F \in \mathcal{F}\}$.
\end{defn}

\begin{remark}
	When there is no risk of ambiguity, we shall write $ad(\mathcal{F})$ instead of $ad^{(X,\tau)}(\mathcal{F})$. Occasionally, we may also write $ad^{\tau}(\mathcal{F})$ or $ad^{X}(\mathcal{F})$ instead of $ad^{(X,\tau)}(\mathcal{F})$.
\end{remark}

\begin{defn}
	If $\mathcal{F}$ is a filter base on a topological space $X$, we say that $\mathcal{F}$ is total if and only if for every filter base $\mathcal{H} \supseteq \mathcal{F}$, we have $ad(\mathcal{H})\neq \emptyset$.
\end{defn}

	Notice that if $\mathcal{F}$ is a filter base on a topological space $X$, the definition of stable under countable intersections filter base does not depend on the topology we consider on $X$. However, whether a filter base $\mathcal{F}$ on $X$ is total or not depends on the topology that we are considering, since being total depends on whether $ad^{X}(\mathcal{G}) \neq \emptyset$ for all $\mathcal{G}$ that extend $\mathcal{F}$.

\begin{defn}(Totally Lindelöf) \label{TotLindDef}
	We say that a topological space $X$ is totally Lindelöf if and only if for every filter base $\mathcal{F}$ on $X$ that is stable under countable intersections, there is a filter base $\mathcal{H} \supseteq \mathcal{F}$ which is total on $X$ and is stable under countable intersections.
\end{defn}

\begin{defn}\label{GdeltaX}
	If $X$ is a topological space, we denote by $G_{\delta}^{X}$ the collection of all countable intersections of open subsets of $X$. Given $S \subseteq X$, we define the following filter:
	$$G_{\delta}(S)= \{G \subseteq X \mid S \subseteq G \textrm{ and } G \in G_{\delta}^{X}\}$$
\end{defn}

As mentioned above, totally Lindelöf is a property that was isolated to unify the proof that countable products of $P$-spaces and countable products of $\sigma$-compact spaces remain Lindelöf.  The proof of the following proposition, which will be used later, may help the reader in gaining some intuition about the definition of totally Lindelöf. For readers interested in learning more about this definition, they should refer to \cite{vaughan}.

\begin{prop} \label{Lw1} $L(\omega_1)$ is totally Lindelöf.
\end{prop}
\begin{proof} Let $\mathcal{F}$ be a stable under countable intersections filter base on $L(\omega_1)$. We first consider the easier case when $\bigcap \mathcal{F} \neq \emptyset$.

\textbf{Case 1}: $\bigcap \mathcal{F} \neq  \emptyset$.

In this case, we fix $p \in \bigcap \mathcal{F}$ and consider the stable under countable intersections filter base $\mathcal{H} = \mathcal{F} \cup \{\{p\}\}$.

It follows that if $\mathcal{G} \supseteq \mathcal{H}$ and $\mathcal{G}$ is a filter base, then $\{p\} \subseteq G$ for all $G \in \mathcal{G}$, hence $ad(\mathcal{G}) \supseteq \bigcap \mathcal{G} \supseteq \{p\} \neq \emptyset$. Thus, $\mathcal{H}$ is total.

\textbf{Case 2}: $\bigcap \mathcal{F} = \emptyset$

We will prove that in this case, $\mathcal{F}$ is total. Given a countable set $C \subseteq \omega_{1}+1$, we fix, for each $\alpha < \omega_{1}$, $F_{\alpha} \in \mathcal{F}$ such that $\alpha \notin F_{\alpha}$. Using the definition of stable under countable intersections, we can fix $F \in \mathcal{F}$ such that $ F \subseteq \bigcap_{\alpha \in C} F_{\alpha}$, therefore $C \cap F = \emptyset$.
Thus, if $\mathcal{G} \supseteq \mathcal{F}$ is a filter base, then for all $G \in \mathcal{G}$, we have $sup(G \cap \omega_{1})=\omega_{1}$. So, for all $G \in \mathcal{G}$, we have $\omega_{1} \in \overline{G}$. Hence, $ad(\mathcal{G}) \supseteq \{\omega_{1}\} \neq \emptyset$. This proves that $\mathcal{F}$ is total.
\end{proof}

We observe that $L(\omega_1)$ is a Lindelöf $P$-space, so one could avoid the computations in Proposition \ref{Lw1} by applying J. Vaughan's result \cite{vaughan}  that every Lindelöf $P$-space is totally Lindelöf.

Our next proposition addresses the equivalence of the definitions of totally Lindelöf:

\begin{prop} \label{stable} 
	Suppose $X$ is a topological space. Then the following are equivalent: 
	\begin{itemize}
		\item[(a)] $X$ is totally Lindelöf. 
		\item[(b)] Every $\omega$-closed filter base $\mathcal{F}$ on $X$ has an extension $\mathcal{H}$ that is a total $\omega$-closed filter base on $X$.
		\item[(c)] Every $\delta$-stable filter base $\mathcal{F}$ on $X$ has an extension $\mathcal{H}$ that is a total $\delta$-stable filter base on $X$. 
	\end{itemize}  
\end{prop}
\begin{proof} We start with (a) $\Rightarrow $ (b). Suppose that $\mathcal{F}$ is an $\omega$-closed filter base on $X$. Then $\mathcal{F}$ is a filter base on $X$ that is stable under countable intersections. By $(a)$, there is filter base $\mathcal{H}\supseteq \mathcal{F}$ on $X$ such that $\mathcal{H} $ is stable under countable intersections and total. Let $$\mathcal{G} = \{ \bigcap S \mid S \in [\mathcal{H}]^{\leq \omega}\}.$$ From the fact that $\mathcal{H}$ is a filter base that is stable under countable intersections, it follows that $\mathcal{G}$ is an $\omega$-closed filter base. Also $\mathcal{G}$ is total, since $\mathcal{H}$ is total.

(b) $\Rightarrow$ (c) Suppose $\mathcal{F}$ is a $\delta$-stable filter base on $X$. Let $$\mathcal{F}'=\{ \bigcap S \mid S \in [\mathcal{F}]^{\leq \omega}\}.$$ From the fact that $\mathcal{F}$ is $\delta$-stable, it follows that $\mathcal{F}'$ is an $\omega$-closed filter base. Let $\mathcal{H} \supseteq \mathcal{F}'$ be a total $\omega$-closed filter base on $X$. Then $\mathcal{H}$ is a total $\delta$-stable filter base that extends $\mathcal{F}$.

(c) $\Rightarrow$ (a) Suppose that $\mathcal{F}$ is a filter base that is stable under countable intersections. Then $\mathcal{F}$ is a $\delta$-stable filter base. By (c), there is $\mathcal{H} \supseteq \mathcal{F}$, which is a total $\delta$-stable filter base. Then $$\mathcal{G}=\{ \bigcap S \mid S \in [\mathcal{H}]^{\leq \omega}\}$$ is a total stable under countable intersections filter base that extends $\mathcal{F}$. 
\end{proof}

For the sake of self-containment, we introduce the definitions of the properties Alster and Frolík.

\begin{defn}(Alster) \label{AlsterDef} 
	If $X$ is a topological space, we say that $\mathcal{G}$ is an Alster covering of $X$ if $\mathcal{G} \subseteq G_{\delta}^{X}$ and for every compact $K \subseteq X$, there is $G \in \mathcal{G}$ such that $K \subseteq G$.
	
	We say that $X$ is an Alster space if every Alster cover has a countable subcover.  
\end{defn}

%\begin{defn}(Frolík) \label{DefFrolík} We say that $X$ is an Alster space if for every Alster covering $\mathcal{G}$ of $X$, there is $\{A_{n} \mid n \in \omega\} \subseteq \mathcal{G}$ such that $X = \bigcup_{n \in \omega}A_{n}$. \end{defn}

\begin{defn}(Frolík) \label{DefFrolík}
We say that $X$ is a Frolík space if\footnote{See \cite[Theorem 12]{FrolikEquiv} for a proof that Definition \ref{DefFrolík} is equivalent to the original definition, where the class of spaces with such property was called $E$. Only later they were referred to as Frolík spaces \cite{tall6}.} there is a sequence of $\sigma$-compact spaces $\langle S_{n} \mid n \in \omega \rangle$  and there are $\varphi$ and $F$ such that $\varphi: X \rightarrow F \subseteq \prod_{n \in \omega}S_{n}$ is a homeomorphism and $F$ is a closed subset of the product space.
\end{defn}

\section{A characterization of regular Alster spaces \label{questiona} }

In this section, we prove Theorem \ref{AlstervsTotthm}. We will begin by proving Lemma \ref{Vaughan51}, which was originally stated in \cite{vaughan} without a proof.

\begin{defn} 
Let $\mathcal{F}$ be a filter base on $X$ and $W \subseteq X$. We say that $\mathcal{F}$ converges to $W$ if for all open sets $U$ that contain $W$, there exists $F \in \mathcal{F}$ such that $F \subseteq U$.   
\end{defn} 

\begin{lemma}(\cite[Lemma 5.1]{vaughan}) \label{Vaughan51}
If $\mathcal{F}$ is a filter base that converges to a compact set, then $\mathcal{F}$ is total.
\end{lemma}

\begin{proof} 
Let $\mathcal{F}$ be a filter base that converges to a compact set $K$ and let $\mathcal{H}$ be a filter base that extends $\mathcal{F}$. Towards a contradiction, suppose that $ad(\mathcal{H})=\emptyset$. For each $p \in K$, let $V_{p}$ be an open set such that for some $H_{p} \in \mathcal{H}$ we have $H_{p} \cap V_{p} = \emptyset$. Since $K$ is compact, there is a finite cover $\{V_{p_{0}}, \cdots,  V_{p_{n}} \} $ of $ K$. Hence, $H=H_{p_{0}} \cap \cdots \cap H_{p_{n}} \in \mathcal{H}$ and $$H \cap (V_{p_{0}}\cup \cdots \cup V_{p_{n}}) = \emptyset.$$ On the other hand, since $\mathcal{F}$ converges to $K$, there is $T \in \mathcal{H}$ such that $T \subseteq (V_{p_{0}}\cup \cdots \cup V_{p_{n}})$, which implies that $H \cap T = \emptyset$, contradicting the fact that $\mathcal{H}$ is a filter base. Hence, $ad(\mathcal{H})$ is non-empty. 
\end{proof} 

We introduce the following definition that will be very helpful when dealing with filter bases:

\begin{defn} \label{Vee} Let $\mathcal{F}$ and $\mathcal{H}$ be filter bases on a set $X$. We denote by $\mathcal{F} \vee \mathcal{H}$ the set $\{F_0 \cap F_1 \mid F_0,F_1 \in \mathcal{F} \cup \mathcal{H} \}$ of pairwise intersections of elements of $\mathcal{F}$ and $\mathcal{H}$. 
    
\end{defn}

Notice that in the Definition \ref{Vee} it is possible that $\mathcal{F} \vee \mathcal{H}$ is not a filter base.

\begin{prop} \label{FvG}
Suppose $\mathcal{F}$ is a stable under countable intersections filter base on a topological space $X$ and $K \subseteq X$ is a compact set. If $$\mathcal{F}\vee G_{\delta}(K)= \{F_0 \cap F_1 \mid F_0, F_1 \in \mathcal{F} \cup G_{\delta}(K) \} $$ is a filter base, then it is a total, stable under countable intersections, filter base.
\end{prop}

\begin{proof} 
If $\mathcal{F}\vee G_{\delta}(K)$ is a filter base, then it is a stable under countable intersections filter base. Given $U$ an open set of $X$ that contains $K$, we have $U \in G_{\delta}(K)$. Therefore, $\mathcal{F}\vee G_{\delta}(K)$ converges to $K$. Hence, by Lemma \ref{Vaughan51}, it follows that $\mathcal{F}\vee G_{\delta}(K)$ is total.
\end{proof}

Next, we shall prove the main result of this section, which will answer Problem \ref{problem1}(a) negatively.

\begin{thm}\label{AlsterImpTotthm} 
If $X$ is an Alster space, then $X$ is totally Lindelöf. 
\end{thm} 

\begin{proof} 
Let $\mathcal{F}$ be a stable under countable intersections filter base on $X$. By Proposition \ref{FvG}, in order to find a total stable under countable intersections filter base extending $\mathcal{F}$, it is enough to find $K \subseteq X$ compact such that $\mathcal{F} \vee G_{\delta}(K)$ is a filter base.
	
Suppose that for every $K \subseteq X$ compact, we have that $\mathcal{F} \vee G_{\delta}(K)$ is not a filter base.

Hence, for each $K \subseteq X$ compact, let $G_{K} \in G_{\delta}(K)$ and $F_{K} \in \mathcal{F}$ be such that $G_{K} \cap F_{K}=\emptyset$. We set $$\mathcal{A}=\{G_{K} \mid K \subseteq X \textrm{ and } K  \mbox{ is compact} \},$$ which is an Alster cover of $X$. 

Since $X$ is an Alster space, there is $\langle K_{n} \mid n \in \omega \rangle$, a countable sequence of compact sets of $X$, such that $ \{G_{K_{n}} \mid n \in \omega\} \subseteq \mathcal{A}$ is a countable subcover of $X$. We have that $G_{K_{n}} \cap F_{K_{n}} = \emptyset$ for every $n \in \omega$, thus 

$$ \displaystyle \left(\bigcap_{n \in \omega}F_{K_n}\right) \cap \left(\bigcup_{n \in \omega} G_{K_{n}}\right) = \emptyset$$

But $X = \bigcup_{n \in \omega} G_{K_{n}}$. Since $\mathcal{F}$ is stable under countable intersections, there is $F \in \mathcal{F}$ such that $F \subseteq \bigcap_{n \in \omega} F_{K_{n}}$, which is a contradiction. 

Thus, there exists some $K \subseteq X$ compact such that $\mathcal{H} = \mathcal{F}\vee G_{\delta}(K)$ is a filter base and from Proposition \ref{FvG}, it is a total stable under countable intersections filter base. Therefore, $X$ is totally Lindelöf. 
\end{proof}

 We conclude this section with two results on filter bases stable under countable intersections. While they are not needed later, these results help clarify how  such filter bases converge in a Lindelöf space.

In Proposition \ref{CharTotal}, below, we prove a  characterization of total stable under countable intersections filter bases that holds when we are restricted to regular totally Lindelöf spaces. 

Our proof of Proposition \ref{CharTotal} relies on the following lemma due to J. Vaughan.
	
	\begin{lemma} \label{Vaughan51comp} \cite[Lemma 5.1]{vaughan} Suppose $X$ is a regular space. If $\mathcal{F}$ is a total filter base on $X$, then there is a non-empty compact set $K\subseteq X$ such that $\mathcal{F}$ converges to $K$.  
\end{lemma} 	

\begin{prop} \label{CharTotal}
Suppose $X$ is a regular totally Lindelöf space and $\mathcal{F}$ is a stable under countable intersections filter base on $X$. Then there is $K$, a compact set such that $$\mathcal{F}\vee G_{\delta}(K)= \{F_0 \cap F_1 \mid F_0, F_1 \in \mathcal{F} \cup G_{\delta}(K) \} $$ is a total stable under countable intersections filter base that extends $\mathcal{F}$. 
\end{prop}

\begin{proof} 
Let $\mathcal{H} \supseteq \mathcal{F}$ be a total stable under countable intersections filter base that extends $\mathcal{F}$. By Lemma \ref{Vaughan51comp}, there is $K \subseteq X$ compact such that $\mathcal{H}$ converges to $K$. We will verify that $\mathcal{F}\vee G_{\delta}(K)$ is a filter base and the proposition will follow from Proposition \ref{FvG}. 

Let $G$ be a $G_{\delta}$ set such that $G \supseteq K$ and let $F \in \mathcal{F}$. 

Write $G = \bigcap_{n \in \omega} V_{n}$ where each $V_n$ is an open set. Since $\mathcal{H}$ converges to $K$, it follows that for every $n \in \omega$, there is $H_n \in \mathcal{H}$ such that $H_n \subseteq V_{n}$. Using the fact that $\mathcal{H}$ is stable under countable intersections, there is $H \subseteq \bigcap_{n \in \omega} H_{n} \subseteq \bigcap_{n \in \omega} V_{n}= G$ such that $H \in \mathcal{H}$. From the fact that $\mathcal{H}$ extends $\mathcal{F}$, we have that $F \cap H \neq \emptyset$, thus $F \cap G \neq \emptyset$. This verifies that $\mathcal{F} \vee G_{\delta}(K)$ is a filter base. 
\end{proof} 

\begin{thm} \label{LindelofFilter}
Suppose $X$ is a Lindelöf space and $\mathcal{F}$ is a stable under countable intersections filter base on $X$. Then $\mathcal{F} $ converges to $ ad(\mathcal{F})$.	
\end{thm}

\begin{proof} 
Suppose that there is an open set $U$ such that $ad(\mathcal{F}) \subseteq U$ and that there is no $F \in \mathcal{F}$ such that $F \subseteq U$. Let $\mathcal{H} = \{ F \setminus U \mid F \in \mathcal{F}\}$ and $\{F_n \mid n \in \omega\} \subseteq \mathcal{F}$. From the hypothesis that $\mathcal{F}$ is stable under countable intersections, we can fix $F \subseteq  \bigcap_{n \in \omega } F_{n}$. Notice that 

$$F \setminus U \subseteq \left(\bigcap_{n \in \omega} F_{n}\right) \setminus U = \bigcap_{n \in \omega} (F_{n } \setminus U)$$

This verifies that $\mathcal{H}$ is stable under countable intersections. Since $X$ is Lindelöf, it follows that $ad(\mathcal{H}) \neq \emptyset$ and $ad(\mathcal{H}) \subseteq (X \setminus U)$. On the other hand, $ad(\mathcal{H}) \subseteq ad(\mathcal{F}) \subseteq U$, a contradiction. 
\end{proof}

	\section{A totally Lindelöf space that is not an Alster space \label{questionb} }

In this section, we prove that the refinement of the usual topology of the real line obtained by taking complements of countable sets is a totally Lindelöf space that is not an Alster space. This will answer Problem \ref{problem1}(b).

\begin{thm} \label{CETHM}
There is a Hausdorff, totally Lindelöf space that is not an Alster space.
\end{thm}
\begin{proof}
Let $\tau$ be the usual topology on the reals. We define $$\rho = \{ U \setminus E \mid U \in \tau \textrm{ and } E \in [\mathbb{R}]^{\leq \omega}\} $$ a refinement of $\tau$.

We will show that $\rho$ is a totally Lindelöf topology that is not Alster. Notice that  every element of $[\mathbb{R}]^{<\omega} $ is a $ G_{\delta}$ set in $\tau$ and there is no countable subset of $[\mathbb{R}]^{<\omega}$ that covers $\mathbb{R}$. In order to prove that $\rho$ is not Alster, we will prove that $[\mathbb{R}]^{<\omega}$ is an Alster cover. This follows from the following claim:

\begin{claim} \label{FiniteClaim}
    Every  compact subset of $(\mathbb{R},\rho)$  is finite.  
\end{claim}

\begin{proof}
    Let $\{p_{n} \mid n \in \omega \}\subseteq \mathbb{R} $ be an infinite set. For each $n \in \omega$, let $U_{n}=\mathbb{R} \setminus \{p_{k} \mid k \geq n\}$ and $\mathcal{O}=\{U_{n} \mid n \in \omega \}$. Then $\mathcal{O}$ is an open cover of $\{p_{n} \mid n \in \omega \}$ that has no finite sub-cover.
\end{proof}

Let $\mathcal{F}$ be a stable under countable intersections filter base on $\mathbb{R}$ such that $\bigcap{\mathcal{F}}=\emptyset$. Since $\mathbb{R}$ is totally Lindelöf, there is $\mathcal{G}\supseteq \mathcal{F}$, a stable under countable intersections filter base that is total in $(\mathbb{R},\tau)$. We will prove that $\mathcal{G}$ is also total in $(\mathbb{R},\rho)$. Let $\mathcal{H}$ be a filter base that extends $\mathcal{G}$.

For each $H \in \mathcal{H}$, we define $V_{H} = \bigcup \{V \in \tau \mid |V \cap H|\leq \omega\}.$ We set $$\widehat{\mathcal{H}} = \{H\setminus V_{H} \mid H \in \mathcal{H}\} \cup \mathcal{H}\supseteq \mathcal{G}.$$ 

Let us verify that $\widehat{\mathcal{H}}$ is a filter base. Given $ H \in \mathcal{H}$ and $E \in [\mathbb{R}]^{\leq \omega}$, using the fact that $\cap \mathcal{F}=\emptyset$, we can fix, for each $z \in E$, a set $F_{z} \in \mathcal{F}$ such that $z \not\in F_{z}$. Since $\mathcal{F}$ is stable under countable intersections, there is $F' \in \mathcal{H}$ such that $ F' \subseteq H \cap \bigcap_{z \in E} F_{z} $, therefore $F'$ is disjoint from $E$.

Since the usual topology on the reals is hereditarily Lindelöf, it follows that $V_{H}$ is Lindelöf. Therefore, we can find $ \mathcal{A} \subseteq \{ V \in \tau \mid |V \cap H|\leq \omega\}$ countable such that $\bigcup \mathcal{A} = V_{H}$. Hence, $V_{H} \cap H$ is countable.

 It follows that for every $H \in \mathcal{H}$, setting $E = H \cap V_{H}$, there is $F' \in \mathcal{F}$ such that $F' \subseteq (H \setminus V_{H}) \subseteq H$. Thus, since $\mathcal{H}$ is a filter base, it follows that $\widehat{\mathcal{H}}$ is a filter base.

We will now prove that for every $H \in \mathcal{H}$, we have $\overline{H \setminus V_{H}}^{\rho}=\overline{H \setminus V_{H}}^{\tau}$. Therefore, from the fact that $\mathcal{G}$ is total in $(\mathbb{R},\tau)$, it will follow that $$\bigcap \{ \overline{T}^{\rho} \mid T \in \widehat{\mathcal{H}} \} = \bigcap \{ \overline{T}^{\tau} \mid T \in \widehat{\mathcal{H}} \} \neq \emptyset.$$

Since $\mathcal{H}$ is an arbitrary filter base extending $\mathcal{G}$, this will imply that $\mathcal{G}$ is total in $X=(\mathbb{R},\rho)$.

The inclusion $\subseteq$ follows from the fact that $\rho$ is a refinement of $\tau$. So we only have to verify the other inclusion. Let $x \in \overline{H \setminus V_{H}}^{\tau}$ and let $V$ be an open set of $(\mathbb{R},\tau)$ such that $x \in V$. If $ |V \cap H| \leq \omega $, then $V \subseteq V_{H}$, which contradicts the choice of $x \in \overline{H \setminus V_{H}}^{\tau}$. Thus, for every open neighborhood $H$ of $x$, we have $$|V \cap H| \geq \omega_{1}.$$ So, for every $E \in [\mathbb{R}]^{\leq \omega}$, we have $$ |V \cap (H \setminus E) | \geq \omega_{1} >0.$$ In particular, if $E = V_{H} \cap H$, we have that $|V \cap (H \setminus V_{H})| \geq \omega_{1}$. Hence, $x \in \overline{H\setminus V_{H}}^{\rho}$, which shows that $\overline{H \setminus V_{H}}^{\rho}=\overline{H \setminus V_{H}}^{\tau}$. 
\end{proof}

 In \cite[Theorem 4]{aurichitallmenger} it was proved that every Alster space is Menger.  Although, for regular spaces, the totally Lindelöf property implies the Aslter property \cite[Proposition 5.1]{noble}, we do not know whether every totally Lindelöf space is Menger.

 By \cite[Proposition 3.1]{ZdomskyMenger}, if there is a Michael space, then every productively Lindelöf space is Menger. In particular, if there is a Michael space, every totally Lindelöf space is Menger. Since it is consistent with $ZFC$ that there is a Michael space \cite{moore1999some}, it follows that $ZFC$ cannot prove that there is a totally Lindelöf Space that is not Menger. 
 
 \begin{question}
     Does the totally Lindelöf property imply the Menger property?
 \end{question}

 The topology $\rho$ constructed above, despite being non-regular, behaves relatively well. As we have shown in Theorem \ref{CETHM}, it is a Hausdorff topology. We will prove that it is  Menger in Proposition \ref{PropMenger} below. 

\begin{prop} \label{PropMenger}
The space $(\mathbb{R},\rho)$ defined above is Menger.
\end{prop}
\begin{proof}
Let $\langle \mathcal{A}_{n} \mid n<\omega \rangle$ be a sequence of open covers in $\rho$. For each $n<\omega$ and each $A \in A_n$, we fix an open set $V(A)$ in the usual topology and a countable subset $E(A)$ such that $A=V(A)\setminus E(A)$.

Since, in the usual topology, $\mathbb{R}$ is a Menger space, we can fix a sequence of finite sets $\langle \mathcal{B}_{k(n)} \mid n<\omega \rangle$ such that for each $n \in \omega$ we have $\mathcal{B}_{k(n)}\subseteq A_{n}$  and $\{V(A) \mid n<\omega \textrm{ and }  A\in\mathcal{B}_{k(n)}\}$ is a covering of $\mathbb{R}$.

Thus, $$\mathbb{R}\setminus\bigcup \{ A \mid n < \omega \textrm{ and } A \in \mathcal{B}_{n} \} \subseteq\bigcup\{E(A)\mid n<\omega \textrm{ and }  A\in \mathcal{B}_{n}\}$$ is countable.

Let $\pi: \omega \rightarrow \bigcup\{E(A)\mid n<\omega \textrm{ and }  A\in \mathcal{B}_{k(n)}\}$ be a bijection. For each $m \in \omega$, since $\mathcal{A}_{m}$ is a covering of $X$, we can fix a set $V(A_{m})\setminus E(A_m)$ in $\mathcal{A}_{m}$  such that $\pi(m) \in V(A_{m}) \setminus E(A_{m})$. If we let $$\mathcal{C}_{m} = \{ V(A) \setminus E(A) \mid A \in \mathcal{B}_{m} \} \cup \{ V(A_{m}) \setminus E(A_{m}) \},$$ then $\mathcal{C}_{m}$ is finite, $\mathcal{C}_{m} \subseteq \mathcal{A}_{m}$ and $$\bigcup_{m \in \omega} \bigcup \mathcal{C}_{m} = X. $$ Thus $\{ \mathcal{C}_{m} \mid m \in \omega \}$ witnesses that $(\mathbb{R},\rho)$ is Menger. 

\end{proof}

Thus, by Theorem \ref{AlsterImpTotthm} and Theorem \ref{CETHM}, we conclude that the totally Lindelöf property lies between the Alster property and the productively Lindelöf property.

	\[\begin{tikzcd}
		\sigma\mbox{-compact} &&&& \mbox{Lindelöf + P-space} \\
		&& \mbox{Alster} \\
		\\
		&& \textbf{Totally Lindelöf} \\
		&& \mbox{Productively Lindelöf}
		\arrow[from=1-1, to=2-3]
		\arrow[from=1-5, to=2-3]
		\arrow[from=2-3, to=4-3]
		\arrow[from=4-3, to=5-3]
	\end{tikzcd}\]

In \cite[Theorem 2]{vaughan}, it is proved that every totally Lindelöf space is productively Lindelöf, then, applying Theorem \ref{CETHM}, we obtain the following Corollary: 
 \begin{cor} \label{PartialAlster}
     There is a non-regular space  $X$ that is productively Lindelöf and it is not an Alster space.
 \end{cor}

In \cite{Alster} and \cite[$\S$7 Question 1]{Raphael} it was asked whether there exists a productively Lindelöf space that is not an Alster space.  Corollary \ref{PartialAlster} provides a partial answer to that question, the following remains open:
\begin{question}
    Is there a regular space that is productively Lindelöf and it is not an Alster space? 
\end{question}
	\section{The discrete set property and Frolík spaces \label{questionc} } 

We will introduce the $\omega_{1}$-discrete set property and prove that no countable product of $\sigma$-compact spaces has this property. 

 Proposition \ref{Prop1} is part of the motivation of the definition of the $\omega_{1}$-discrete set property. We left the proof of Proposition \ref{Prop1} to the reader. 

\begin{prop} \label{Prop1}
Let $X$ be a $\sigma$-compact space. Then there is no sequence $\{x_{\alpha} \mid \alpha < \omega_{1}\} \subseteq X$ such that for all $\beta < \omega_1$, $\{x_{\alpha} \mid \alpha < \beta \}$ is a closed discrete set.
\end{prop}
%\begin{proof} Let $X = \bigcup_{n \in \omega}K_{n}$, where for each $n \in \omega$,  $K_{n}$ is compact. Towards a contradiction, suppose that there is  $\{x_{\alpha} \mid \alpha < \omega_{1}\} \subseteq X$ such that for all $\beta < \omega_1$,  $\{x_{\alpha} \mid \alpha < \beta \}$ is a closed discrete set. There exists $n^{*} \in \omega$ such that  $ \{x_{\alpha} \mid \alpha < \omega_{1} \} \cap K_{n^{*}} $ is uncountable. Suppose, towards a contradiction, that for every $\beta < \omega_{1}$,  $\{x_{\alpha} \mid \alpha < \beta \} \cap K_{n*} $ is finite of cardinality $m(\beta) \in \omega$. Then, there exists $m^* \in \omega$ such that $m(\beta)=m^*$ for uncountable many $\beta < \omega_1$. However,  $\beta < \gamma < \omega_{1}$ implies $ \{x_{\alpha} \mid \alpha < \beta \} \cap K_{n^{*}} \subseteq \{x_{\alpha} \mid \alpha < \gamma \} \cap K_{n^{*}}$. Then, it follows that for every $\beta < \omega_{1}$ the set $\{x_{\alpha} \mid \alpha < \beta \} \cap K_{n^{*}}$ has $m^*$ elements and $\{x_{\alpha} \mid \alpha < \omega_1 \} \cap K_{n^{*}}$ is finite, which is a contradiction. Hence, there is $\beta < \omega_1$ such that $\{x_{\alpha} \mid \alpha < \beta \} \cap K_{n^{*}}$ is a countable closed discrete set, contradicting the fact that $K_{n^{*}}$ is compact.  \end{proof}

\begin{defn}\label{dsp}
If $X$ is a topological space and there is $\{x_{\alpha} \mid \alpha < \omega_{1}\} \subseteq X$ such that for all $\beta < \omega_1$, the set $\{x_{\alpha} \mid \alpha < \beta \}$ is closed and  discrete, we say that $X$ has the $\omega_1$-discrete set property. 
\end{defn}

\begin{thm}\label{Prop3}
Let $\{X_n\mid n<\omega\}$ be a sequence of $\sigma$-compact spaces.
Then the product space $\prod_{n<\omega}X_n$ does not have the $\omega_{1}$-discrete set property.
\end{thm}
\begin{proof}
Firstly, let us establish the following notation: for each $m<\omega$, $\pi_m:\prod_{n<\omega}X_n\to X_m$ denotes the usual projection.

Let $\{x_\alpha\mid\alpha<\omega_1 \}\subseteq\prod_{n<\omega}X_n$ be such that $x_\alpha=x_\beta$ if and only if $\alpha=\beta$.
For $\alpha<\omega_1$ and $n<\omega$, define $x_\alpha(n)=\pi_n(x_\alpha)$.
For each $n<\omega$, fix a sequence $\{K_n^i\mid i<\omega\}$ where $K_n^i$ is a compact subspace of $X_n$, $K_n^i\subseteq K_n^{i+1}$ and $X_n=\bigcup_{i<\omega}K_n^i$ for each $i<\omega$. 

We will verify that for some $\beta < \omega_{1}$, the set $\{x_{\alpha} \mid \alpha < \beta \}$ has an accumulation point. 

Given $\alpha < \omega_{1}$, we have $x_{\alpha} \in \prod_{n < \omega}X_{n}$ and there is some $i < \omega$ such that $x_{\alpha}(0) \in K^{i}_{0}$. Therefore, $x_{\alpha} \in \pi^{-1}_{0}[K^{i}_{0}]$. Thus,  

$$\{x_{\alpha} \mid \alpha < \beta \} \subseteq \bigcup_{i<\omega}\pi_0^{-1}[K_0^i]$$

and there is a least natural number $i_0$ such that 
\[|\{x_{\alpha} \mid \alpha < \omega_{1}\}\cap\pi_0^{-1}[K_0^{i_0}]|=\aleph_1.\]

Inductively, if there are $i_0,...,i_n<\omega$ such that
\[|\{x_{\alpha} \mid \alpha < \omega_{1}\}\cap\pi_0^{-1}[K_0^{i_0}]\cap...\cap\pi_n^{-1}[K_n^{i_n}]|=\aleph_1\]
it follows that there is a least natural $i_{n+1}$ such that

		\[(|(\{x_{\alpha} \mid \alpha < \omega_{1}\}\cap\pi_0^{-1}[K_0^{i_0}]\cap...\cap\pi_n^{-1}[K_n^{i_n}])\cap\pi_{n+1}^{-1}[K_{n+1}^{i_{n+1}}]|=\aleph_1.\]

For each $n<\omega$ define $F_n=\{x_{\alpha} \mid \alpha < \omega_{1}\}\cap\bigcap_{m\leq n}\pi_m^{-1}[K_m^{i_m}]$.
Since each $F_n$ is uncountable, fix a strictly increasing sequence $\{\alpha_n\mid n<\omega\}\subseteq\omega_1$ such that for each $n<\omega$, $x_{\alpha_n}\in F_n$.

Since $K_0^{i_0}$ is compact, if $\{x_{\alpha_{n}}(0) \mid n \in \omega \}$ is infinite, then we can fix an accumulation point $p_0\in K_0^{i_0}$  of $\{x_{\alpha_{n}}(0) \mid n \in \omega \}$. If $\{x_{\alpha_{n}}(0) \mid n \in \omega \}$ is not infinite, we can fix $h \in \omega $ such that $\{ n \in \omega \mid x_{\alpha_{n}}(0) = x_{\alpha_{h}}  \}$ is infinite. In this case, we let $p_{0} = x_{\alpha_{h}}(0)$. In both cases, for every neighborhood $V$ of $p_0$, the set $\{n<\omega \mid x_{\alpha_n}(0)\in V\}$ is infinite and we have that $$\{n < \omega \mid x_{\alpha_n}(0) \in V\} \subseteq \{ n < \omega \mid x_{\alpha_{n}} \in \pi^{-1}_{0}[V]\}.$$

\begin{claim} There exists a sequence $\langle p_{n} \mid n \in \omega \rangle \in \prod_{n \in \omega} X_{n}$ such that for every $n \in \omega $, if $V_{m}$ is a neighborhood of $p_m$ in $X_m$ for each $m \leq n$, then the set $\{k<\omega\mid x_{\alpha_k} \in\bigcap_{m\leq n}\pi_m^{-1}[V_m]\}$ is infinite.
\end{claim} 

\begin{proof} 
We will construct $\langle p_{n} \mid n < \omega \rangle$ by recursion. We have discussed the case $n=0$ above, so let us assume that $n+1<\omega$ and that we have constructed $\langle p_m\mid m\leq n\rangle\in \prod_{m\leq n}K_m^{i_m}$ such that if we fix a neighborhood $V_m$ of $p_m$ in $X_m$ for each $m \leq n$, then the set $\{k<\omega\mid x_{\alpha_k} \in\bigcap_{m\leq n}\pi_m^{-1}[V_m]\}$ is infinite.

Suppose the induction step does not hold, namely, for all $p\in K_n^{i_n}$ and $m< n$, there is a neighborhood $V_m(p)$ of $p_m$ and $V(p)$ of $p$ such that $$\{k<\omega\mid x_{\alpha_k}\in\bigcap_{m<n}\pi_m^{-1}[V_m(p)]\cap\pi_n^{-1}[V(p)]\}$$ is finite. 
Since $K_n^{i_n}$ is compact, let us fix $l \in \omega $ and $ \{q_0, \cdots, q_l \} \subseteq K_n^{i_n}$  such that $K_n^{i_n}\subseteq\bigcup_{j \leq l} V(q_j)$.
For each $m<n$, let $V_m=\bigcap\{V_m(p)\mid p\in I\}$. 

Note that $\prod_{k < \omega} X_k \subseteq \bigcup_{j \leq l} \pi^{-1}_{n+1}[V(q_{j})]$. Hence, for each $k < \omega $, if $x_{\alpha_{k}} \in \bigcap_{m < n+1}\pi_m^{-1}[V_m]$ there is some $j \leq l$ such that 
$$x_{\alpha_{k}} \in \bigcap_{m < n+1}\pi_m^{-1}[V_m] \cap \pi^{-1}_{n+1}[V(q_{j})] $$

Thus, \begin{equation} \label{eq1} \bigcup_{j \leq l } \{ k < \omega \mid x_{\alpha_{k}} \in \bigcap_{m < n+1}\pi_m^{-1}[V_m] \cap \pi^{-1}_{n+1}[V(q_{j})] \} \end{equation}  is a subset of
$$\bigcup_{j \leq l } \{ k < \omega \mid x_{\alpha_{k}} \in \bigcap_{m < n+1}\pi_m^{-1}[V_m] \cap \pi^{-1}_{n+1}[V(q_{j})] \}$$
which is a contradiction. Indeed,  for each $j \leq l$ we have that
$$\{ k < \omega \mid x_{\alpha_{k}} \in \bigcap_{m < n+1}\pi_m^{-1}[V_m] \cap \pi^{-1}_{n+1}[V(q_{j})] \}$$
is a subset of
$$\{ k < \omega \mid x_{\alpha_{k}} \in \bigcap_{m < n+1}\pi_m^{-1}[V_{m}(q_j)] \cap \pi^{-1}_{n+1}[V(q_{j})] \}$$
which is finite. Therefore, the union in (\ref{eq1}) is finite. However, we assumed that $\{k<\omega\mid x_{\alpha_k}\in\bigcap_{m<n}\pi_m^{-1}[V_m]\}$ is infinite. This contradiction proves the claim.
\end{proof}
		
		Now, let us fix $\langle p_{n} \mid n \in \omega \rangle \in\prod_{n<\omega}X_n$ as given by the claim above.    
		
		Let $\beta=\sup_{n<\omega}\alpha_n$. We have that $\{x_{\alpha} \mid \alpha < \beta \} $ is not a discrete set, therefore,  $X$ does not have the $\omega_{1}$-discrete set property.
	\end{proof}

    It is worth noting that, even if we restrict the statement of Theorem \ref{Prop3} to finite products, it is not possible to weaken the hypothesis that the spaces are $\sigma$-compact  to just Lindelöf.   To illustrate this, we shall use a space famous for being a counterexample to the preservation of topological properties by products: the Sorgenfrey Line, denoted here as $\mathbb{R}_s$. For a reference on $\mathbb{R}_{s}$ see \cite{Engelking77}.

\begin{prop}  Every discrete subset of $\mathbb{R}_s$ is at most countable and there is a closed discrete subset $A \subseteq \mathbb{R}^{2}_{s}$ such that $|A|=2^{\aleph_0}$. In particular,    $\mathbb{R}_s$ does not have the  $\omega_1$-discrete set property whereas $\mathbb{R}_s^2$ does. 
\end{prop}

\begin{proof} The Sorgenfrey Line  is hereditarily Lindelöf, therefore there is no uncountable discrete subspace of $\mathbb{R}_{s}$. On the other hand, $A=\{\po{x}{-x}: x \in \mathbb{R}_s\}$ is a closed discrete subset of $\mathbb{R}_{s}\times \mathbb{R}_{s}$ such that $|A|=2^{\aleph_{0}}$.
    
 %Consider a subset $A = \{x_\alpha \mid  \alpha < \omega_1 \}$ of $\mathbb{R}_s$. Assume that, for every $\alpha < \omega_1 $, there exists an $\epsilon_\alpha > x_\alpha$ such that $[x_\alpha, \epsilon_\alpha) \cap A = \{x_\alpha\}$. By their construction, the intervals $\{(x_\alpha, \epsilon_\alpha) \mid \alpha < \omega_1\}$ would be disjoint open sets in $\mathbb{R}$, which contradicts $\mathbb{R}$ being separable. Therefore, there exists a $\gamma < \omega_1$, such that $x_\gamma$ is an accumulation point of $A$.
 
 %Since $\mathbb{R}_s$ is first countable, there exists a sequence $\{x_{\beta_n} \mid n \in \omega\} \subseteq A$ which converges to $x_\gamma$. Let $\theta = \sup (\{\gamma\} \cup \{\beta_{n} \mid n \in \omega\}) + 1$. We have that $\{x_\alpha \mid \alpha < \theta\}$ is not a discrete set. Therefore, as we intended to demonstrate, $\mathbb{R}_s$ does not have an uncountable discrete subset. 

%Now consider the subset $A=\{\po{x}{-x}: x \in \mathbb{R}_s\}$ of $\mathbb{R}^2_s$. The set $A$ is a closed discrete subspace of size $\mathfrak{c}$. In particular,  $\mathbb{R}^2_s$ has the $\omega_1$-discrete set property.

\end{proof}
	
	\begin{defn}  \label{Pspace} We say that a topological space $X$ is a $P$-space if every countable intersection of open sets in $X$ is an open set in $X$. 	
\end{defn}

\begin{cor}\label{corFrolik}
	There is a P-space that is not Frolík, namely $L(\omega_1)$.
\end{cor}
\begin{proof}
	If there were a homeomorphism to a closed subspace $Z$ of the product space
	\[\rho: L(\omega_1) \to Z \subseteq \prod_{n<\omega}X_n\]
	where each $X_n$ is $\sigma$-compact, then $\{\rho(\alpha)\mid\alpha<\omega_{1}\}$ would witness that $\prod_{n \in \omega} X_{n}$ has the $\omega_1$-discrete set property, contradicting Theorem \ref{Prop3}.
\end{proof}

\begin{prop}\label{AlsterMetrizable}
    $\omega^{\omega}$ is a Frólik space that is not a totally Lindelöf space.
\end{prop}
\begin{proof}
    It is immediate that $\omega^{\omega}$ is Frólik. Towards a contradiction suppose that $\omega^{\omega}$ is a totally Lindelöf space. Since $\omega^{\omega}$ is metrizable,  by \cite[Corollary 6.2]{vaughan}, it follows that  $\omega^{\omega}$ is $\sigma$-compact which is a contradiction.
\end{proof}

\begin{remark}
    In \cite{Alster}, it was proved that every metrizable Alster spaces is $\sigma$-compact. Thus,  in the proof of Proposition \ref{AlsterMetrizable}, we could have applied  \cite{Alster} instead of \cite{vaughan}.
\end{remark}

Our next corollary together with Proposition \ref{AlsterMetrizable} will  answer Problem\ref{problem1}(c):

\begin{cor} \label{corc} $L(\omega_1)$ is an Alster space that is not Frolík.
\end{cor}
\begin{proof}
 It follows, by Proposition \ref{Lw1}, that $L(\omega_{1})$ is totally Lindelöf. It is straightforward to verify that $L(\omega_{1})$ is regular. Hence, by Theorem \ref{AlstervsTotthm}, we conclude that $L(\omega_{1})$ is Alster. Thus, $L(\omega_{1})$ is totally Lindelöf, Alster and, by Corollary \ref{corFrolik}, $L(\omega_1)$ is not a Frolík space.
\end{proof}
\section{Acknowledgments} 

We are grateful to  Lúcia R. Junqueira and Norman Noble for their helpful suggestions and comments on earlier versions of this manuscript.  We are grateful to Rodrigo Carvalho and Artur Tomita for several discussions on totally Lindelöf spaces and  Lúcia R. Junqueira for discussions on non-regular topological spaces. We also would like to thank the referee for their thorough reading and suggestions, which contributed to the improvement of our manuscript.

Some of the results of this paper were presented by the first author at  the \textit{56th Spring Topology and Dynamical Systems Conference}, March 2023. He thanks the organizers for the opportunity to speak and the participants for their feedback.

	\providecommand{\abntreprintinfo}[1]{%
 \citeonline{#1}}
\setlength{\labelsep}{0pt}
	
	\Addresses

\end{document}